\documentclass[11pt]{amsart}

\usepackage[paperwidth=18.4cm, paperheight= 26cm,
            body={14true cm, 21true cm},
            ]{geometry}

\usepackage{amssymb}
\usepackage{url}
\usepackage{amscd}
\usepackage{texdraw}
\usepackage{color}
\usepackage{graphicx}
\usepackage{titletoc}

\theoremstyle{plain}
\newtheorem{thm}{Theorem}[section]

\numberwithin{equation}{section}

\def \R{\Bbb{R}}
\def \D{\Bbb{D}}

\def \I{\Bbb{I}}
\def\ra{\rightarrow}

\def\pa{\partial}

\def\e{\epsilon}

\def\l{\lambda}
\def\s{\sigma}

\def\O{\Omega}
\def\Si{\Sigma}

\input txdtools

\begin{document}

\title[Beltrami system and 1-QC embeddings]
{Beltrami system and 1-quasiconformal embeddings in higher dimensions}
\author{Huanhuan Yang, Tao Cheng and Shanshuang Yang}

\address{Huanhuan Yang: Department Of Mathematics, Shantou University,
	Shantou, Guangdong, 515063, People's Republic of China}
\email{huan2yang@stu.edu.cn}

\medskip
\address{Tao Cheng: Department Of Mathematics,
	East China Normal University,
	Shanghai, 200241,
	People's Republic of China}
\email{tcheng@math.ecnu.edu.cn}

\medskip
\address{Shanshuang Yang: Department Of Mathematics and Computer Sciences,
Emory University,
Atlanta, GA 30322, U.S.A}
\email{syang@mathcs.emory.edu}

\thanks{The second author is partially supported by National Natural Science Foundation of China (No.11371268 and No.11471117) and Science and Technology Commission of Shanghai Municipality  (No.13dz2260400). The third author is partially supported by National Natural Science Foundation of China (No.11471117) and by PERS of Emory University. }


\begin{abstract}
In this paper we derive necessary and sufficient conditions for a smooth surface in $\R^{n+1}$ to admit a local
1-quasiconformal parameterization by a domain in $\R^n$ ($n\geq 3$). We then apply these conditions to specific
hypersurfaces such as cylinders, paraboloids and ellipsoids. As a consequence, we show that the classical Liouville theorem
about the rigidity of 1-quasiconformal maps between domains in $\R^n$ with $n\geq 3$ does not extend to embeddings of domains into
a higher dimensional space.
\end{abstract}


\maketitle
\section{Introduction}

The celebrated Liouville Theorem about 1-quasiconformal mappings  states that
if $D$ is a domain in the Euclidean space $\R^n$ ($n\geq 3$), then an embedding
$f: D\ra\R^n$ is 1-QC if and only if it is the restriction to $D$ of a M\"obius
transformation of $\bar\R^n$ (see \cite{Ge1, Ge2} and \cite{Re1, Re2}).
A M\"obius transformation is a finite composition of reflections in spheres or hyperplanes.
This deep result spells out the fundamental difference between conformal
mapping theory in the plane and in higher dimensions. Its sophisticated proof
has a long history and involves tools from analysis, geometry and PDE (see
\cite{IM}, Chapter 5 for more details). This result is also useful in
differential geometry. In particular, it is a major ingredient in the proof
of Mostow's rigidity on compact hyperbolic manifolds of higher dimensions
\cite{Mo}. It is also used in \cite{Ho} to prove a stronger version of Kuiper's theorem on conformally flat manifolds \cite{Ku}.

This paper is largely motivated by the question whether the above mentioned Liouville Theorem for 1-QC mappings can be extended to 1-QC embeddings from
a domain in $\R^n$ into $\R^{n+1}$.  Recall that
an embedding $f: X\ra Y$  in metric spaces $X$ and $Y$
is called {\it quasiconformal}, abbreviated QC, if there is a
constant $K<\infty$ such that
\begin{equation}
\limsup_{r\ra 0}
\frac{\sup\{|f(x)-f(y)|: |x-y|=r\}}{\inf\{|f(x)-f(y)|: |x-y|=r\}}
\leq K
%\tag{1.1}
\end{equation}
for all $x\in X$. In this case we also say $f$ is $K$-QC.
Here $|x-y|$ denotes the distance between $x$ and $y$ in a metric space.
This type of distortion condition plays an important role in recent development
of analysis on general metric spaces (see \cite{He}).

In the case $n=2$, it was shown in \cite{Ya} that any $C^{1+\alpha}$-smooth surface $\Si$ in $\R^3$ given by
$\Si=\{(x,y,z) : (x,y)\in D\subset \R^2, z=\phi(x,y)\}$ admits a differentiable $1$-QC parameterization by a plannar domain.
This implies that there are $1$-QC embeddings $f: \R^2 \ra \R^3$ which are not M\"obius transformations. The proof of this result
depends heavily on the theory of $2$-dimensional Beltrami system, which is well established. This paper is devoted to the
study of Beltrami systems and $1$-QC embeddings of higher dimensions.

In Section 2, we establish equivalent conditions for smooth 1-QC embeddings from differential geometric and algebraic points of view. Section 3 is devoted to the connection between the existence of $1$-QC parameterizations of a surface and the existence of
solutions to a Beltrami system of partial differential equations. Section 4 verifies the conformal flatness of a cylinder and a
paraboloid in $\R^{n+1}$ with $n\geq 3$ by computing the Weyl conformal curvature (when $n\geq 4$) or the Cotton curvature (when $n=3$) and applying a classical result of Weyl and Schouten. Finally, in Section 5 we  investigate the conformal flatness of a hypersurface
from a differential geometric point of view by computing the principal curvatures. As examples, we reconfirm that cylinders and
paraboloids are conformally flat and show that an ellipsoid is not conformally flat when $n\geq 4$.
In summary, we conclude that there are local
smooth $1$-QC embeddings $f: D\subset\R^n \ \ra \ \R^{n+1}$ which are not M\"obius for all $n\geq 3$. Thus the
Liouville Theorem does not hold for embeddings of a domain in $\R^n$ into $R^{n+1}$ as well.

\section{Equivalent definitions for smooth 1-QC maps}
In this section, we characterize smooth 1-QC embeddings using languages  from differential geometry and linear algebra.
Let $f: \O \ra \R^m$ be a smooth embedding of a domain $\O\subset\R^n$ into $\R^m$ with $m\geq n\geq 2$. For a fixed point
$a\in \O$, denote the differential of $f$ at $a$ by $\D f(a)$ or $\frac{df(a)}{dx}$. If we write $f$ as $f(x)=(f_1, f_2,\cdots, f_m)$,
then the differential has the following matrix representation:
$$\D f(a)=\left(\frac{\pa f_j(a)}{\pa x_i}\right)_{m\times n},$$
which can also be regarded as a linear transformation from $\R^n$ to $\R^m$. We say that the embedding $f$ is orientation-preserving if
$$\text{det}(\D^tf(a)\D f(a))>0$$
for each $a\in\O$, where $\D^tf(a)$ is the transpose of $\D f(a)$ and $\text{det}(\cdot)$ is the determinant of a matrix.

\begin{thm}
Let $f: \O \ra \R^m$ be an orientation-preserving smooth embedding of a domain $\O\subset\R^n$ into $\R^m$ with $m\geq n\geq 2$. Then the following statements are equivalent:

(a) $f$ is 1-QC in the sense of definition (1.1);

(b) for each $p\in\O$ and all unit vectors $e_1, e_2\in \R^n$,
$$\|\D f(p)e_1\|=\|\D f(p)e_2\|;$$

(c) for  each $p\in\O$  and all $v_1, v_2\in \R^n$,
$$\langle \D f(p)v_1,\D f(p)v_2\rangle = \l(p)^2\langle v_1,v_2\rangle,$$
where
$$\l(p)=[\text{det}(\D^tf\D f)]^{\frac{1}{n}};$$
(d) $f$ is a solution to the equation
\begin{equation}
\D^tf(x)\D f(x)=\l(x)\I,
\end{equation}
where $\l(x)$ is as in (c) and $\I$ is the identity matrix.
\end{thm}

{\bf Remarks.} The above equivalent conditions are more or less known to experts, may be in different forms. Due to the lack of a
precise reference, we record them here and give a short proof for the completeness. We also note that, according to the proof, the result is valid for just once differentiable embeddings.

{\bf Proof of Theorem 2.1.}
For the proof of (a) $\Rightarrow$ (b), without loss of generality, we may
fix $p=0$ and assume $f(p)=0$. Given two unit vectors $e_1, e_2\in\R^n$,
let $x=te_1$ and $y=te_2$ for small $t>0$. By the differentiability of $f$ at $0$, it follows that
$$\lim_{t\ra 0}\frac{|f(x)|}{|x|}=|\D f(0)e_1| \ \text{and} \ \lim_{t\ra 0}\frac{|f(y)|}{|y|}=|\D f(0)e_2|.$$
Thus, (b) follows from (1.1) with $K=1$.

For the proof of (b) $\Rightarrow$ (a), we fix $p\in\O$.
Since $f$ is differentiable, for any $x\in\O$ in a neighborhood of $p$
we have
$$f(x)-f(p)=\D f(p)(x-p)+\e(x,p),$$
where $\e(x,p)/|x-p|\ra 0$  as $|x-p|\ra 0$. Therefore, for small $r>0$ and $x, y\in\O$ with $|x-p|=|y-p|=r$,
it follows that
\begin{equation}
\begin{aligned}
\frac{|\D f(p)(y-p)|-|\e(y,p)|}{|\D f(p)(x-p)|+|\e(x,p)|}
&\leq\frac{|f(y)-f(p)|}{|f(x)-f(p)|}\\
&\leq\frac{|\D f(p)(y-p)|+|\e(y,p)|}{|\D f(p)(x-p)|-|\e(x,p)|}.
\end{aligned}\end{equation}
Letting $r\ra 0$, (2.2) together with condition (b) yields that
$$\limsup_{r\ra 0}
\frac{\sup\{|f(x)-f(p)|: |x-p|=r\}}{\inf\{|f(x)-f(p)|: |x-p|=r\}}=1$$
and this shows that $f$ is 1-QC at each point $p$.

\medskip
The equivalence of (b), (c) and (d) are elementary results from linear algebra about non-degenerate linear
transformations $T: \R^n \ra \R^m$. For details, we refer the reader to any standard textbook on this subject.

\section{1-QC parameterization and Beltrami system}

In this section, we first derive a necessary and sufficient condition for a surface in $\R^m$ to admit a $1$-QC parameterization
in terms of the higher dimensional Beltrami system. Then we recall a classical result of Weyl and Schouten on the existence of
local solutions to the Beltrami system.

\begin{thm}
Let $\Si$ be a surface in $R^m$ with a diffeomorphic parameterization $\s: D \ra \Si$ by a domain $D\subset \R^n$, $m>n\geq 2$.
Then $\Si$ admits a differentiable $1$-QC parameterization $f: \O \ra \Si$ by a domain $\O\subset\R^n$
if and only if the Beltrami system
\begin{equation}
\left(\frac{dz}{dx}\right)^t \left(\frac{dz}{dx}\right)=G(x)
\end{equation}
has a differentiable solution $z=h(x)$ in $D$, where $G(x)$ is the $n\times n$ matrix determined by the given parameterization
$\s$ as follows:
\begin{equation}
G(x)=\left(\frac{d\s}{dx}\right)^t\left(\frac{d\s}{dx}\right).
\end{equation}
\end{thm}

\begin{proof}
Fix a surface $\Si$ in $\R^m$ with a  parameterization $\s: D \ra \Si$ by a domain $D\subset \R^n$ as in Theorem 3.1.
First assume that $\Si$ admits a differentiable $1$-QC parameterization $f: \O \ra \Si$ by a domain $\O\subset\R^n$.
Consider the composition map
$$y=f^{-1}(\s(x)): D \ra \O.$$
By the chain rule and the inverse function theorem, it follows that
$$\begin{aligned}
\left(\frac{dy}{dx}\right)^t \left(\frac{dy}{dx}\right)
&=\left(\frac{d\s}{dx}\right)^t\left(\frac{dy}{d\s}\right)^t\left(\frac{dy}{d\s}\right)\left(\frac{d\s}{dx}\right)\\
&=\left(\frac{d\s}{dx}\right)^t\left[\left(\frac{df}{dy}\right)^t\left(\frac{df}{dy}\right)\right]^{-1}\left(\frac{d\s}{dx}\right).
\end{aligned}
$$
Since $f: \O \ra \Si$ is $1$-QC, Theorem 2.1 (d) yields that
$$\left(\frac{df}{dy}\right)^t\left(\frac{df}{dy}\right)=\l(y)\I,$$
where
$$\l(y)=\left[det\left(\left(\frac{df}{dy}\right)^t\left(\frac{df}{dy}\right)\right)\right]^{\frac{1}{n}}.$$
Thus it follows that
\begin{equation}
 \left(\frac{dy}{dx}\right)^t \left(\frac{dy}{dx}\right)
=\l(y)^{-1}\left(\frac{d\s}{dx}\right)^t\left(\frac{d\s}{dx}\right)=\l(y)^{-1}G(x).
\end{equation}

To show that the Beltrami system (3.1) has a solution in $D$, let $z=z(y)$ be a differentiable solution to the Cauchy-Riemann
system
$$\left(\frac{dz}{dy}\right)^t\left(\frac{dz}{dy}\right)=\l(y)\I$$
in $\O$. By (3.3) we have
$$\begin{aligned}
\left(\frac{dz}{dx}\right)^t \left(\frac{dz}{dx}\right)
&=\left(\frac{dy}{dx}\right)^t\left(\frac{dz}{dy}\right)^t\left(\frac{dz}{dy}\right)\left(\frac{dy}{dx}\right)\\
&=\l(y) \left(\frac{dy}{dx}\right)^t \left(\frac{dy}{dx}\right)=G(x).
\end{aligned}
$$
This shows that $z=h(x)=z(f^{-1}(\s(x)))$ is a solution to the Beltrami system (3.1) in $D$.

\medskip
Conversely, assume that the Beltrami system (3.1)  has a smooth solution $z=h(x)$ in $D$. Let
$\O=h(D)$ and $f(z)=\s(h^{-1}(z))$. Then, it follows from (3.1) and (3.2) that
$$\begin{aligned}
\left(\frac{df}{dz}\right)^t \left(\frac{df}{dz}\right)
&=\left(\frac{d\s}{dx}\frac{dx}{dz}\right)^t\left(\frac{d\s}{dx}\frac{dx}{dz}\right)\\
&=\left(\frac{dx}{dz}\right)^tG(x) \left(\frac{dx}{dz}\right)=\I.
\end{aligned}
$$
By Theorem 2.1, this shows that
$$\s=f(z)=\s(h^{-1}(z)): \O \ra \Si$$
is a differentiable $1$-QC parameterization of $\Si$ by the domain $\O$ in $\R^n$ as desired.
\end{proof}

\medskip
Theorem 3.1 reveals that finding a differentiable $1$-QC parameterization of a surface $\Si\subset \R^m$  by a domain $\O$ in $\R^n$ boils down to solving the corresponding Beltrami system (3.1).
Unlike the $2$-dimensional case where the Beltrami system is well understood, when dimension $n\geq 3$ the Beltrami system (3.1)
is highly overdetermined and very little is known about the existence of global solutions in a given domain.
However, for the existence of local solutions, we have the following classical results of Weyl and Schouten (see \cite{IM}, Theorem 2.7.1) and \cite{Sc}.

\begin{thm}
Let $G(x)$ be a smooth matrix function, in a domain $D\subset\R^n$, valued in the space of symmetric positive definite
$n\times n$ matrices. Then the Beltrami system
\begin{equation}
\left(\frac{df}{dx}\right)^t \left(\frac{df}{dx}\right)=G(x)
\end{equation}
has local non-constant solutions in $D$ if and only if the following conditions are satisfied.\\
(a) When $n=3$, the Cotton tensor vanishes: $C_{ijk}=0$, $i, j, k=1, 2, 3$.\\
(b) When $n\geq 4$, the Weyl conformal curvature tensor vanishes: $W_{ijkl}=0$, $i, j, k, l=1, 2, \cdots, n$.
\end{thm}

\medskip
Using the terminology from differential geometry, if a symmetric positive definite matrix $G(x)$ (or a metric) satisfies condition
(a) or (b) in the above Theorem, it
is called {\it conformally flat}. The corresponding surface (or manifold) $\Sigma$ is also called conformally flat.
In order to apply the above Theorems to specific surfaces, we recall the definitions of various involved tensors and curvatures as follows. Given a smooth matrix function $G(x)$ as above, let $g_{ij}$ and $g^{ij}$ denote the elements of $G(x)$ and the elements of the inverse $G^{-1}(x)$, respectively.
With the conventional Einstein's summation notation, the various curvatures and tensors associated with $G(x)$ can be
defined and computed as follows (see \cite{Le}, Chapter 7 and \cite{IM}, Chapter 2).
The Weyl conformal curvature tensor is defined as
\begin{equation}\label{weyleq}
W_{ijkl} = R_{ijkl} + \frac{1}{n-2}(g_{ki}R_{jl}-g_{il}R_{jk}+g_{jl}R_{ik}-g_{jk}R_{il})
+ \frac{1}{(n-1)(n-2)}(g_{il}g_{jk}-g_{ki}g_{jl})R,
\end{equation}
where
$R_{ijkl}$ is the (0,4)-Riemann curvature tensor $R_m: \overset{4}{\bigotimes}\mathcal{T}(\Sigma)\to \mathbb{R}$
$$R_m(X,Y,Z,W) = \langle\nabla_X\nabla_Y Z - \nabla_Y\nabla_X Z - \nabla_{[X,Y]} Z, W\rangle$$
%$$R_{ijkl}=\frac{1}{2}(\pa_j\pa_lg_{ik}+\pa_i\pa_kg_{jl}-\pa_i\pa_lg_{jk}-\pa_j\pa_kg_{il})$$
in local coordinates, and $$R_{ij} = g^{\mu\nu}R_{\mu ij\nu} \ \mbox{and} \   \  R = g^{ij}R_{ij}$$
are the Ricci curvature and scalar curvature, respectively.
The Schouten tensor $S_{ij}$ and Cotton tensor $C_{ijk}$ can be defined as
\begin{equation}
S_{ij} = \frac{1}{n-2}\Big(R_{ij}-\frac{R}{2(n-1)}g_{ij}\Big)
\end{equation}
and
\begin{equation}
C_{ijk} = \nabla_jS_{ik}-\nabla_kS_{ij},
\end{equation}
respectively, where the covariant derivative of Schouten tensor can be computed through
\begin{equation}
\nabla_jS_{ik} = \frac{\partial S_{ik}}{\partial x_j} - S_{mk}\Gamma^m_{ij} - S_{im}\Gamma^m_{kj},
\end{equation}
and $\Gamma^m_{ij}$ is the Christoffel symbol of $G$:
$$\Gamma^m_{ij}=\frac{1}{2}g^{km}(\pa_ig_{kj}+\pa_jg_{ki}-\pa_kg_{ij}).$$

\section{Hypersurfaces in $\R^{n+1}$}

In this section, we determine the conformal flatness of certain hypersurfaces in $\R^{n+1}$. This is done by computing
the corresponding tensors of a surface and then applying above theorems given in the previous section.

\subsection{The Riemannian metric on a surface}
Let $\Si$ be a hypersurface in $\R^{n+1}$ with a smooth parameterization
$$\s=\s(x_1,\cdots, x_n) = (x_1,\cdots, x_n, r(x_1, \cdots, x_n)): D \ra \Si $$
by a domain $D\subset \R^n$.
The standard basis $\{\partial_1,\cdots,\partial_n\}$ for the tangent space $T_p\Sigma$ at each point $p$ is given by
$$\partial_i =\frac{\partial\sigma}{\partial x_i}  = (0, \cdots, 0, 1,0,\cdots, 0, r_i) \quad (i = 1, \cdots, n), $$
where $r_i$ denotes the partial derivative $\partial_{x_i}r$.
So the unit normal vector is
$$\displaystyle N=\frac{(-r_1,-r_2, \cdots, -r_n, 1)}{\sqrt{r_1^2+\cdots+r_n^2+1}}.$$
%and $\sigma_{ij} = (0, \cdots,0, \frac{\partial^2f}{\partial x_i\partial x_j}) = (0, \cdots,0, r_{i,j}).$
The Riemannian metric on $\Sigma$ induced by the Euclidean metric is then determined by the matrix
\begin{equation}\label{gMatr}
G(x)=\big(g_{ij}\big) = \big(\partial_i \cdot \partial_j\big) = \left( \begin{matrix}
1+r_1^2 & r_1r_2 & \cdots &r_1r_n\\
r_1r_2 & 1+r_2^2 & \cdots & r_2r_n \\
\vdots & \vdots & \ddots & \vdots \\
r_1r_n & r_2r_n & \cdots & 1+r_n^2
\end{matrix} \right),
\end{equation}
which is also called the {\it first fundamental form} of $\Si$.

\subsection{Second fundamental form and shape operator}
In general it is very complicated to explicitly compute the curvature tensors. In the hypersurface case,
the second fundamental form and shape operator provide useful tools in computing the relevant curvature tensors
efficiently.

Following \cite{Le}, Chapter 8, when $\Sigma$ is a Riemannian submanifold of $\widetilde{\Sigma}$, its second fundamental form is defined by the map $II$ from the tangent bundle $\mathcal{T}(\Sigma)\times\mathcal{T}(\Sigma)$ to the normal bundle $\mathcal{N}(\Sigma)$ as:
$$II(X,Y) = (\widetilde{\nabla}_{X}Y)^\perp$$
with $\widetilde{\nabla}$ being the Riemannian connection on  $\widetilde{\Sigma}$. For a hypersurface in particular, since the codimension is one we can then replace $II$ by a scalar quantity using the normal vector $N$ to trivialize $\mathcal{N}(\Sigma)$. That is,
$$II(X,Y) = h(X,Y)N.$$
The term $h$ is classically called the {\it scalar second fundamental form} of $\Sigma$. By the definition,
\begin{equation}\label{hcomp}
\displaystyle h(\partial_i, \partial_j) = \widetilde{\nabla}_{\partial_i}\partial_j \cdot N =\frac{\partial^2\sigma}{\partial x_i\partial x_j} \cdot N = \frac{r_{i,j}}{\sqrt{r_1^2+\cdots+r_n^2+1}},
\end{equation}
where $r_{i,j}=\pa_i\pa_jr$ is the partial derivative.
Raising one index of $h$, one can define the {\it shape operator} $s$ of $\Sigma$:
$$\langle sX,Y\rangle = h(X,Y)\quad \forall X,Y\in\mathcal{T}(\Sigma).$$
At each point $p \in \Sigma$ , the shape operator $s$ is a self-adjoint
linear transformation on $T_p(\Sigma)$, its eigenvalues $\kappa_1,\cdots,\kappa_n$ are called the {\it principal curvatures} of $\Sigma$ at $p$.
The concept of shape operator provides a convenient way for computing principal curvatures and other curvature tensors.

\subsection{Computing Riemann curvature tensor by shape operator} Using shape operator $s$, the Riemann curvature tensor
of a hypersurface $\Si$ can be computed as follows.

First, by the Weingarten equation,
\begin{equation}\label{sdieq}
s\partial_i = -\partial_{x_i}N.
\end{equation}
Denoting $s\partial_i = s_i^j\partial_j$, we have
\begin{equation}\label{sdiE}
s\partial_i = (s_i^1, \cdots, s_i^n, \sum_js_i^jr_j).
\end{equation}
Let $b=r_1^2+\cdots+r_n^2+1$. Combining (\ref{sdieq}) and (\ref{sdiE}), we then obtain
\begin{equation}\label{sij1}
s_i^j = -b^{-\frac{3}{2}}r_j\sum_kr_kr_{k,i}+b^{-\frac{1}{2}}r_{j,i}.
\end{equation}

Next, one can also write the shape operator $s$ in the matrix form:
 $$h(\partial_i, \partial_k)=h_{ik} = \langle\mathit{s}\partial_i, \partial_k\rangle = \langle s_i^j\partial_j, \partial_k\rangle =s_i^jg_{jk}.$$
So we get $s_i^j = g^{jk}h_{ik}$, i.e.
$$(\mathit{s}) =(g^{-1})(h).$$
The inverse $(g^{jk})$ of the Riemannian metric matrix $G(x)$ can be computed directly from $g$ in (\ref{gMatr}) by some elementary but tedious work in liner algebra. Alternatively, if one rewrites $s_i^j$ in (\ref{sij1}) as
$$s_i^j = \sum_k\big(-b^{-\frac{3}{2}}r_jr_k +b^{-\frac{1}{2}}\delta_{jk})r_{k,i},$$
using $s_i^j = g^{jk}h_{ik}$ and the symmetry of $h$, one can easily derive that
\begin{equation}
g^{jk} = \delta_{jk} - b^{-1}r_jr_k,
\end{equation}
where $\delta_{ij}$ is the usual Kronecker symbol.

Finally, by the Gauss equation, one derives
$$
R_m(\partial_i,\partial_j,\partial_k,\partial_l) = h(\partial_i,\partial_l)h(\partial_j,\partial_k)-h(\partial_i,\partial_k)h(\partial_j,\partial_l).
$$
Thus, we arrive at the following explicit and practical formula for computing the Riemann curvature tensor:
\begin{equation}\label{rijkl}
R_{ijkl} = h_{il}h_{jk}-h_{ik}h_{jl} = b^{-1}(r_{i,l}r_{j,k}-r_{i,k}r_{j,l}).
\end{equation}

\subsection{Flatness of a cylinder}
Consider the hypersurface $\Si=\mathbb{R}^{n-1} \times S^1$ in $\R^{n+1}$ given by
$$ \sigma : \mathbb{R}^{n-1}\times[-1,1] \to \mathbb{R}^{n-1} \times S^1,~ \sigma(x_1,\cdots, x_n) = \sigma (x_1,\cdots,x_n, \sqrt{1-{x_n}^2}).$$
Here we have $r = \sqrt{1-{x_n}^2}$.  Hence $r_i=0$ for all $i$ except for $i=n$ and the induced matrix $G(x)$ has the form
$$G(x)=\left( \begin{matrix}
1 & 0 & \cdots & 0\\
0 & 1 & \cdots & 0 \\
\vdots & \vdots & \ddots & \vdots \\
0 & 0 & \cdots & 1+r_n^2
\end{matrix} \right).$$
Therefore $r_{i,j}$ all vanish except for $r_{n,n}$. By expression (\ref{rijkl}), one can easily see that
$$R_{ijkl} = 0, \quad\forall  i,j,k,l = 1, \cdots, n.$$
Thus, it follows from (3.5) and (3.7) that  both the Weyl conformal curvature $W_{ijkl}$ and the Cotton tensor $C_{ijk}$ vanish.
In consequence, Theorem 3.2 yields that for all $n\geqslant 3$ the hypersurface $\mathbb{R}^{n-1} \times S^1$ is conformally flat.

\subsection{Flatness of a paraboloid}
Consider the hypersurface $\sigma :\mathbb{R}^n \to \Sigma \subseteq \mathbb{R}^{n+1}$ with
$$\sigma(x_1,\cdots, x_n) = (x_1,\cdots, x_n, r(x_1, \cdots, x_n)), \quad r = x_1^2 + \cdots + x_n^2.$$
In this case $ r_i = 2x_i, r_{i,j}=2\delta_{ij}$. The associated matrices are given by
$$G(x)=(g_{ij}), \  G^{-1}(x)=(g^{ij})$$
with
$$ g_{ij} = \delta_{ij}+4x_ix_j, \   g^{ij} = \delta_{ij}-4b^{-1}x_ix_j,$$
where
$$b=r_1^2+r_2^2+\cdots+1=4x_1^2+\cdots+4x_n^2+1.$$

To compute the Weyl conformal curvature by definition (3.5), we first note that, by (4.7),  the  Riemannian curvature tensor $R_{ijkl}$ can be computed as
$$ R_{ijkl} = 4b^{-1}(\delta_{il}\delta_{jk}-\delta_{ik}\delta_{jl}).$$
Thus, it follows that
$$
R_{ij} = g^{\mu\nu}R_{\mu ij\nu}
= \sum_{\mu\neq i,\nu\neq j}4b^{-1}(\delta_{\mu\nu}-4b^{-1}x_\mu x_\nu)(\delta_{\mu \nu}\delta_{ij}-\delta_{\mu j}\delta_{i\nu}).
$$
When $i=j$,
$$ \begin{aligned}
R_{ii} &= \sum_{\mu\neq i,\nu\neq i}4b^{-1}(\delta_{\mu\nu}-4b^{-1}x_\mu x_\nu)\delta_{\mu \nu}\\
&= \sum_{\mu\neq i}4b^{-1}(1-4b^{-1}x_\mu^2)
= 4(n-2)b^{-1}+4b^{-2}+4b^{-2}\cdot 4x_i^2.
\end{aligned}$$
When $i\neq j$,
$$\begin{aligned}
R_{ij}& = -\sum_{\mu\neq i,\nu\neq j}4b^{-1}(\delta_{\mu\nu}-4b^{-1}x_\mu x_\nu)\delta_{\mu j}\delta_{i\nu}\\
&= -4b^{-1}(\delta_{ji}-4b^{-1}x_j x_i)
= 4b^{-2}\cdot 4x_ix_j.\end{aligned}
$$
Therefore, for all $i, j $ we have
\begin{equation}
R_{ij} = \big(4(n-2)b^{-1}+4b^{-2}\big)\delta_{ij} + 4b^{-2}\cdot 4x_ix_j.
\end{equation}
Denote by $c=4(n-2)b^{-1}+4b^{-2}$ the coefficient of $\delta_{ij}$ in the above equation.  We then have
\begin{eqnarray*}
R &=& g^{ij}R_{ij} \\
  &=& \big(\delta_{ij}-4b^{-1}x_ix_j\big)\big(c\delta_{ij} + 4b^{-2}\cdot 4x_ix_j \big)\\
  &=& c\delta_{ij} + 4b^{-2}\cdot 4x_ix_j\delta_{ij} -4cb^{-1}x_ix_j\delta_{ij} - 4b^{-3}(4x_ix_j)^2 \\
  &=& nc + 4b^{-2}(b-1)-cb^{-1}(b-1)-4b^{-3}(b-1)^2\\
  &=& 4(n-1)(n-2)b^{-1}+8(n-1)b^{-2}.
\end{eqnarray*}
We note that throughout the computation the following indentities are used:
$$\sum_i4x_i^2=b-1,  \qquad  \sum_{i,j}(4x_ix_j)^2 = (b-1)^2.$$

Next we compute the following two expressions used in the definition of  the Weyl tensor $W_{ijkl}$:
\begin{eqnarray*}
&\ &g_{ki}R_{jl}-g_{il}R_{jk}+g_{jl}R_{ik}-g_{jk}R_{il} \\
&=& (\delta_{ki}+4x_kx_i)\big( c\delta_{jl} + 4b^{-2}\cdot 4x_jx_l\big)-(\delta_{il}+4x_ix_l)\big( c\delta_{jk} + 4b^{-2}\cdot 4x_jx_k\big)\\
& &+(\delta_{jl}+4x_jx_l)\big( c\delta_{ik} + 4b^{-2}\cdot 4x_ix_k\big)-(\delta_{jk}+4x_jx_k)\big( c\delta_{il} + 4b^{-2}\cdot 4x_ix_l\big)\\
&=& 2c(\delta_{ki}\delta_{jl}-\delta_{il}\delta_{jk}) + (c+4b^{-2})(4x_kx_i\delta_{jl}-4x_ix_l\delta_{jk}+4x_jx_l\delta_{ik}-4x_jx_k\delta_{il});
\end{eqnarray*}
and
\begin{eqnarray*}
g_{il}g_{jk}-g_{ki}g_{jl} &=& (\delta_{il}+4x_ix_l)(\delta_{jk}+4x_jx_k) - (\delta_{ki}+4x_kx_i)(\delta_{jl}+4x_jx_l)\\
&=& \delta_{il}\delta_{jk}-\delta_{ki}\delta_{jl}+4x_jx_k\delta_{il}
+4x_ix_l\delta_{jk}-4x_jx_l\delta_{ki}-4x_kx_i\delta_{jl}.
\end{eqnarray*}

Finally, putting all pieces together into expression (\ref{weyleq}) yields that
\begin{eqnarray*}
W_{ijkl} &=& 4b^{-1}(\delta_{il}\delta_{jk}-\delta_{ik}\delta_{jl})+ \frac{2c}{n-2}(\delta_{ki}\delta_{jl}-\delta_{il}\delta_{jk}) \\
&&+\frac{c+4b^{-2}}{n-2}(4x_kx_i\delta_{jl}
-4x_ix_l\delta_{jk}+4x_jx_l\delta_{ik}-4x_jx_k\delta_{il})\\
&&+\frac{R}{(n-1)(n-2)}(\delta_{il}\delta_{jk}-\delta_{ki}\delta_{jl}+4x_jx_k\delta_{il}
+4x_ix_l\delta_{jk}-4x_jx_l\delta_{ki}-4x_kx_i\delta_{jl})\\
&=& \Big(4b^{-1}-\frac{2c}{n-2}+\frac{R}{(n-1)(n-2)}\Big)(\delta_{il}\delta_{jk}-\delta_{ik}\delta_{jl}) \\
&&+\Big(\frac{c+4b^{-2}}{n-2}-\frac{R}{(n-1)(n-2)}\Big)(4x_kx_i\delta_{jl}-4x_ix_l\delta_{jk}+4x_jx_l\delta_{ik}-4x_jx_k\delta_{il}) \\
&= &0.
\end{eqnarray*}
This shows that when $n\geq 4$ the paraboloid is conformally flat.

When $n=3$, we need to evaluate the Cotton tensor in order to verify the conformal flatness. In this example,
\begin{eqnarray*}
S_{ij} &=& \frac{1}{n-2}\Big(R_{ij}-\frac{R}{2(n-1)}g_{ij} \Big) \\
 & =& \frac{1}{n-2}\Big(c\delta_{ij}+4b^{-2}\cdot 4x_ix_j - \frac{R}{2(n-1)}(\delta_{ij}+4x_ix_j) \Big)\\
  & =& \frac{1}{n-2}\Big((c - \frac{R}{2(n-1)})\delta_{ij}+(4b^{-2}-\frac{R}{2(n-1)})4x_ix_j) \Big)\\
  & =& 2b^{-1}(\delta_{ij}-4x_ix_j).
\end{eqnarray*}
The Christoffel symbols can be evaluated as
\begin{eqnarray*}
\Gamma^i_{kl} &=& \frac{1}{2}g^{im}\big(\frac{\partial g_{mk}}{\partial x_l} + \frac{\partial g_{ml}}{\partial x_k} - \frac{\partial g_{kl}}{\partial x_m}\big) \\
 & =& \frac{1}{2}(\delta_{im}-4b^{-1}x_ix_m)\cdot 8x_m\delta_{kl} \\
 &=& 4b^{-1}x_i\delta_{kl}.
\end{eqnarray*}

Therefore, the Cotton tensor $C_{ijk}$ can be computed as follows:
\begin{eqnarray*}
C_{ijk} &=&\nabla_jS_{ik}-\nabla_kS_{ij}\\
 &=&\frac{\partial S_{ik}}{\partial x_j} - \frac{\partial S_{ij}}{\partial x_k} + S_{mj}\Gamma^m_{ik} - S_{mk}\Gamma^m_{ij} \\
&=&-2b^{-2}\cdot 8x_j(\delta_{ik}-4x_ix_k)+2b^{-1}(-4x_i\delta_{jk}-4x_k\delta_{ij}) \\
&&+2b^{-2}\cdot 8x_k(\delta_{ij}-4x_ix_j)+2b^{-1}(-4x_i\delta_{jk}-4x_j\delta_{ik}) \\
&&+2b^{-1}(\delta_{mj}-4x_mx_j)\cdot 4b^{-1}x_m\delta_{ik}
-2b^{-1}(\delta_{mk}-4x_mx_k)\cdot 4b^{-1}x_m\delta_{ij} \\
&=&(-16b^{-2}+8b^{-1})(x_j\delta_{ik}-x_k\delta_{ij}) \\
&& + 8b^{-2}x_j\delta_{ik}-8b^{-2}x_j\delta_{ik}(b-1) - 8b^{-2}x_k\delta_{ij}-8b^{-2}x_k\delta_{ij}(b-1)\\
&=& 0.
\end{eqnarray*}
Since the Cotton tensor vanishes, we conclude that the 3-dimensional paraboloid hypersurface is also conformally flat.

\section{The principal curvature approach to conformal flatness}

In theory, given a surface $\Si\subset\R^{n+1}$, one can compute the matrix $G(x)$  and its Weyl curvature tensor
$W_{ijkl}$ (or the Cotton tensor $C_{ijk}$ when $n=3$)  by using (3.5) (or (3.7)) to determine the conformal flatness of $\Si$.
In reality, however, the computation involved can be cumbersome and close to impossible, even for fairly simple surfaces such as an
ellipsoid. Fortunately, when $n\geq 4$, one can appeal to the following principal curvature approach (see \cite{Sc} and \cite{NM}).

\begin{thm}
When $n\geqslant 4$, a Riemannian manifold $M^n$ is conformally flat if and only if at least n-1 of the principal curvatures coincide at each point.
\end{thm}

We shall use this approach to reconfirm  the flatness of a cylinder or a paraboloid and verify the non-flatness of an ellipsoid.
Recall that the principal curvatures of $\Si$ are the eigenvalues of the shape operator $s$ defined in the previous section.

\subsection{Principal curvatures of cylinders and paraboloids}
Let $\Si$ be a cylinder or paraboloid in $\R^{n+1}$ ($n\geq 4$) defined as above. In the case of a cylinder,
the normal vector is
$$N = (0, \cdots, 0, x_n,  \sqrt{1-{x_n}^2}).$$
Thus we have $\partial _k N = 0$ for $ k =1, \cdots, n-1.$ Therefore,
$$s\partial _k = -\partial_k N = 0 = 0\cdot \partial _k \ \mbox{ for} \  k =1, \cdots, n-1.$$
 This means that $s$ has eigenvalue 0 with multiplicity at least $n-1$. Therefore, at each point, $\mathbb{R}^{n-1} \times S^1$ has at least $n-1$ common principal curvatures 0. It deduces that the cylinder $\Si=\mathbb{R}^{n-1} \times S^1$ is conformally flat as $n\geqslant 4$.

In the case of a paraboloid, the shape operator $s$ has the following matrix representation
$$
(s) = (g^{-1})(h) = 2b^{-\frac{1}{2}} (g)^{-1}.
$$
The metric matrix is
$$
(g) = I_n + 4\left( \begin{matrix}
x_1^2 & x_1x_2 & \cdots &x_1x_n\\
x_1x_2 & x_2^2 & \cdots & x_2x_n \\
\vdots & \vdots & \ddots & \vdots \\
x_1x_n & x_2x_n & \cdots & x_n^2
\end{matrix} \right).
$$
Because the eigenvalues of the matrix $(x_ix_j)_{n\times n}$ are $0, \cdots, 0, x_1^2+\cdots+x_n^2$, the eigenvalues of $(g)$ are therefore $1, \cdots, 1, b.$ Finally, the eigenvalues of $(\mathit{s})$ are
$$
2b^{-\frac{1}{2}}, \cdots, 2b^{-\frac{1}{2}}, 2b^{-\frac{3}{2}},
$$
which are also the principal curvatures of the paraboloid.
Since this  paraboloid has $n-1$ repeating principal curvatures, we conclude that when $n\geqslant 4$ this paraboloid is conformally flat.

\subsection{Non-flatness of an ellipsoid}
Consider the hypersurface $\Sigma \subseteq \mathbb{R}^{n+1}$ defined by the parametrization
$$\sigma(x_1,\cdots, x_n) = (x_1,\cdots, x_n, r(x_1, \cdots, x_n)), \quad r = \sqrt{1-(a_1x_1^2 +\cdots + a_nx_n^2)}.$$ Here we assume that $0 < a_i \neq 1$ and at least two $a_i$'s  are distinct. Under the above established notation, we have
$$\displaystyle r_i = \frac{-a_ix_i}{r}, ~r_{i,j} = \frac{-a_i\delta_{ij} r+a_ix_ir_j}{r^2} =\frac{-a_i\delta_{ij} - r_ir_j}{r}.$$ Therefore,
$$h_{ij} = b^{-\frac{1}{2}}r_{i,j} = -b^{-\frac{1}{2}}r^{-1}(a_i\delta_{ij} - r_ir_j).$$
Let $q = -b^{-\frac{1}{2}}r^{-1}$. The shape operator in matrix form is
then given by
\begin{eqnarray*}
(\mathit{s})& =& (\mathit{g}^{-1})(\mathit{h}) \\
&=&[I_n + (r_ir_j)_{n\times n}]^{-1}\cdot [diag(a_i)+(r_ir_j)_{n\times n}]\cdot q \\
&=&[I_n + (r_ir_j)_{n\times n}]^{-1}\cdot[I_n +(r_ir_j)_{n\times n}+diag(a_i-1)]\cdot q \\
&= &qI_n + [I_n + (r_ir_j)_{n\times n}]^{-1}\cdot[diag(\frac{1}{a_i -1})]^{-1}\cdot q \\
&=& qI_n +[diag(\frac{1}{a_i-1})+(\frac{1}{a_i-1}r_ir_j)_{n\times n}]^{-1} \cdot q.
\end{eqnarray*}
Let
$$B = diag(\frac{1}{a_i-1})+(\frac{1}{a_i-1}r_ir_j)_{n\times n}.$$
To compute the  eigenvalues of $B$, we observe that $B-\lambda I_n$ can be rewritten as $A+uv^t$, where
$$
A =diag(\frac{1}{a_i-1}- \lambda), ~u =[\frac{1}{a_1-1}r_1, \cdots, \frac{1}{a_n-1}r_n]^t, ~v=[r_1, \cdots,r_n]^t.
$$
By the relation $det(A+uv^t)=(1+v^tA^{-1}u)detA$ from linear algebra, we have
$$det(B-\lambda I_n) = (1+v^tA^{-1}u)detA = \prod\limits_{i=1}^n (\frac{1}{a_i-1} - \lambda)\cdot (1+v^tA^{-1}u).$$
Furthermore, since
$$\displaystyle v^tA^{-1}u = \sum\limits_{i=1}^n r_i^2(\frac{1}{a_i-1}-\lambda)^{-1}(a_i-1)^{-1}
= \sum\limits_{i=1}^n\frac{r_i^2}{1-(a_i-1)\lambda},$$
it follows that
$$det(B-\lambda I_n) =\prod\limits_{i=1}^n (\frac{1}{a_i-1} - \lambda)\cdot \big(1+\sum\limits_{i=1}^n\frac{r_i^2}{1-(a_i-1)\lambda}\big).$$
Choose specific values for the parameters $a_i$ as follows:
$$a_1-1 = \cdots =a_{n-1}-1 = 1,  \ a_n -1 = \frac{1}{2}.$$
Then
$$\begin{array}{ccl}
det(B-\lambda I_n) &=& (1-\lambda)^{n-1}(2-\lambda)\big(1+\frac{\sum\limits_{i=1}^{n-1} r_i^2}{1-\lambda} + \frac{2r_n^2}{2-\lambda}\big)\\
&=& (1-\lambda)^{n-2}\big((1-\lambda)(2-\lambda)+\sum\limits_{i=1}^{n-1} r_i^2(2-\lambda)+2r_n^2(1-\lambda)\big).
\end{array}$$
Since $\lambda = 1$ is not a root of $(1-\lambda)(2-\lambda)+\sum\limits_{i=1}^{n-1} r_i^2(2-\lambda)+2r_n^2(1-\lambda)$,  the matrix $B$ can not have $n-1$ repeating eigenvalues. Neither does $(s) = qI_n + qB^{-1}$. That is, the ellipsoid doesn't have $n-1$ repeating principal curvatures. In conclusion, this ellipsoid is not conformally flat when $n\geq 4$.

\subsection{Concluding remarks}
Consider a $1$-quasiconformal embedding $f: D\subset\R^n \ \ra \R^m$ with $m\geq n\geq 2$. The classical Liouville theorem says that if $m=n\geq 3$, then $f$ is a M\"obius transformation. It was shown in \cite{Ya} that this rigidity result does not hold for
$m>n=2$. The above examples of cylinders and paraboloids show that this Liouville type rigidity result does not hold for the
higher dimensional case $m>n\geq 3$ as well. Note that  when $n=2$, as shown in \cite{Ya}, a paraboloid in $\R^3$ admits a global $1$-QC parameterization by a plannar domain. When $n\geq 3$, however, the above example only shows that a paraboloid in
$\R^{n+1}$ locally admits a $1$-QC parameterization. It is still an open question whether there is a global $1$-QC map between
such a paraboloid and a domain in $\R^n$. Furthermore, as illustrated above, an $n$-dimensional ellipsoid (with $n\geq 4$) is not
even locally $1$-QC equivalent to a domain in $\R^n$. This is quite contrary to the fact that a $2$-dimensional ellipsoid is globally
$1$-QC equivalent to the unit $2$-sphere. When $n=3$ the computation of the Cotton tensor for an ellipsoid is prohibitively complicated. So it has not yet been determined in this way whether a three-dimensional ellipsoid is conformally flat, let along the
global conformal equivalence to the unit three-sphere.

\end{document}